\documentclass[amssymb,reqno,amsfonts,refcheck,12pt,verbatim,righttag]{amsart}

\usepackage{graphicx}
\usepackage{graphics,color}
\usepackage{amssymb,mathrsfs}
\usepackage{graphicx,color,tikz,caption,subcaption}
\usepackage[all]{xy}
\usepackage{mathtools}
\usepackage{enumerate}
\usepackage{amsmath}
\usepackage{bbm}
\usepackage[colorlinks, linkcolor=blue,citecolor=blue, anchorcolor=blue,urlcolor=blue,hypertexnames=false]{hyperref}

\numberwithin{equation}{section} 
\newtheorem{thm}{Theorem}[section]
\newtheorem{lem}[thm]{Lemma}
\newtheorem{pro}[thm]{Proposition}
\newtheorem{cor}[thm]{Corollary}

\newcommand{\R}{\mathbb{R}}

\newcommand{\N}{\mathbb{N}}

\usepackage{float}
\begin{document}
\baselineskip 14pt
\title{Metrical properties of the large products of partial quotients  in continued fractions}
\author{Bo Tan and Qing-Long Zhou$^{\dag}$}
\address{$^{1}$ School  of  Mathematics  and  Statistics,
                Huazhong  University  of Science  and  Technology, 430074 Wuhan, PR China}
          \email{tanbo@hust.edu.cn}

\address{$^{2}$ School of Science, Wuhan University of Technology, 430070 Wuhan, PR China }
\email{zhouql@whut.edu.cn}
\thanks{$^{\dag}$ Corresponding author.}


\date{}

\begin{abstract}
The study of products of consecutive partial quotients in the continued fraction arises naturally out of the  improvements to Dirichlet's theorem.
Achieving some variant forms of metrical theory in continued fractions,  we study the distribution of the at least two large products of partial quotients among the first $n$ terms. More precisely, writing  $[a_1(x),a_2(x),\ldots]$ the continued fraction expansion of an irrational number $x\in(0,1)$, for  a non-decreasing function $\varphi\colon \N\to\R$, we completely determine the size of the set
\begin{align*}
\mathcal{F}_2(\varphi)=\Big\{x\in[0,1)\colon  \exists 1\le k\neq l \le n,  ~&a_{k}(x)a_{k+1}(x)\ge \varphi(n), \\&a_{l}(x)a_{l+1}(x)\ge \varphi(n) \text{ for infinitely many } n\in \N  \Big\}
\end{align*}
in terms of Lebesgue measure and Hausdorff dimension.
\end{abstract}

\maketitle
{\small \tableofcontents}

\section{Introduction}
	Every irrational number $x\in[0,1)$ can be uniquely expressed as a simple infinite continued fraction expansion of the following form$\colon$
\begin{equation}\label{cf}
x =\frac{1}{a_{1}(x)+\frac{1}{a_{2}(x)+\frac{1}{a_{3}(x)+\ddots}}}:=[a_1(x),a_{2}(x),a_{3}(x),\ldots].
\end{equation}
This expansion can be induced by the Gauss map $T\colon [0,1) \to [0,1)$ defined as
\begin{equation*}
T(0)=0,~ T(x)= \frac{1}{x}\!\!\!\pmod1~\text{for} ~x\in(0,1),
\end{equation*}
with $a_{n}(x)=\lfloor\frac{1}{T^{n-1}(x)}\rfloor$ (here $\lfloor\cdot\rfloor$ denotes the greatest integer less than or equal to a real number and $T^0$ denotes the identity map).	
For each $n\ge1,$ $a_n(x)$ is called the $n$-th partial quotient of $x.$ Further, we write the finite truncations
$$[a_1(x), a_{2}(x), \ldots, a_{n}(x)]:=\frac{p_n(x)}{q_n(x)}$$
to be the $n$-th convergent to $x.$

The metrical theory of continued fractions  concerns the quantitative study of properties of partial quotients for  almost all $x\in [0,1)$.
A well studied object is the metric properties of the classical set of $\psi$-approximable numbers
$$\left\{x\in[0,1)\colon a_{n+1}(x)\ge \frac{1}{q_{n}^{2}(x)\psi\big(q_{n}(x)\big)} \text{ for infinitely many } n\in \mathbb{N} \right\},$$
where $\psi\colon [1,\infty) \to \mathbb{R}$ is a non-increasing  positive function. The theorems of
Khintchine and Jarn\'{\i}k, the fundamental results in Diophantine approximation, were both obtained by studying the above set.
A real number $x$ is $\psi$-approximable if the partial quotients in its continued fraction expansion grow fast.
In other words, the growth speed of the partial quotients reveals how well a real number can be approximated by rationals.
The well-known Borel-Bernstein theorem \cite{B11,B12}, an analogue of Borel-Cantelli `0-1' law with respect to Lebesgue measure for the set of real numbers with large partial quotients,
states that the Lebesgue measure of the set
$$\mathcal{E}_1(\varphi):=\{x\in[0,1)\colon a_{n}(x)\ge \varphi(n) \text{ for infinitely many } n\in \mathbb{N}\}$$
is either 0 or 1 according to  the convergence or divergence of the series $\sum_{n=1}^{\infty}\varphi(n)$ respectively.
The Hausdorff dimensions of $\mathcal{E}_1(\varphi)$  for different kind of functions $\varphi$ were also obtained:
Good \cite{G41} considered the  polynomial case $\varphi(n)=n^{\alpha}$ with $\alpha\ge0$;
{\L}uczak \cite{L97} solved the super-exponential case $\varphi(n)=a^{b^{n}}$ with $a, b>1$;
Wang and Wu \cite{WW08} determined completely the Hausdorff dimension of the set $\mathcal{E}_1(\varphi)$ for an arbitrary function $\varphi.$
\begin{thm} [Wang-Wu, \cite{WW08}]\label{WW}
Let $\varphi\colon \mathbb{N}\to \mathbb{R}$ be a  positive function. Write
$$\log B=\liminf_{n\to\infty}\frac{\log \varphi(n)}{n} \text{ and } \log b=\liminf_{n\to\infty}\frac{\log\log \varphi(n)}{n}.$$
\begin{itemize}
\item [\rm(1)] When $B=1,$ $\dim_{H}\mathcal{E}_1(\varphi)=1.$

\item [\rm(2)] When $B=\infty,$ $\dim_{H}\mathcal{E}_1(\varphi)=1/(1+b).$

\item [\rm(3)]  When $1<B<\infty,$ $$\dim_{H}\mathcal{E}_1(\varphi)=\inf\big\{s\ge0\colon \mathsf{P}(T, -s\log B-s\log |T'|)\le0\big\},$$
\end{itemize}
where $T'$ denotes the derivative of $T$, $\mathsf{P}$ denotes the pressure function (see the definition in  Subsection \ref{pressure}), and $\dim_H$ denotes the Hausdorff dimension.
\end{thm}

The existence of large partial quotients also destroys some limit theorems in continued fractions.
For example, writing $\mathcal{S}_n(x):=\sum_{k=1}^{n}a_{k}(x),$
Khintchine \cite{K35} proved that $\frac{\mathcal{S}_n(x)}{n\log n}$ converges in measure to $\frac{1}{\log 2}$ with respect to the Lebesgue measure $\mathcal{L}$,
while Philipp \cite{P88} showed that there is not a suitable regular function $\varphi\colon \mathbb{N}\to \mathbb{R}_{+}$ such that $\frac{\mathcal{S}_n(x)}{\varphi(n)}$ almost everywhere converges to a finite nonzero constant. However, Diamond and Vaaler proved in \cite{DV86} that, if the largest $a_{k}(x)$ is discarded from $\mathcal{S}_n(x),$ then for almost every $x\in[0,1),$
$$\lim_{n\to\infty}\frac{\mathcal{S}_n(x)-\max\{a_{k}(x)\colon 1\le k\le n\}}{n\log n}=\frac{1}{\log 2}.$$
A key step to obtain the above result is to show that the numbers with two large partial quotients among the first $n$ terms are rare for $n$ large enough. More precisely, define
\begin{align*}
\mathcal{F}_1(\varphi):=\big\{x\in[0,1)\colon \exists 1\le k\neq l\le n, ~&a_{k}(x)\ge \varphi(n), \\
                                                                                                            &a_{l}(x)\ge \varphi(n) \text{ for infinitely many } n\in \mathbb{N}\big\}.\end{align*}
Diamond and Vaaler \cite{DV86} proved that $\mathcal{L}\big(\mathcal{F}_1(\varphi)\big)=0$ if $\varphi(n)=n(\log n)^{c}$ with $c>1/2$ (here and subsequently $\mathcal L$ denotes the Lebesgue measure of a Borel set).
Tan, Tian and Wang \cite{TTW22} characterized completely the size of the set $\mathcal{F}_1(\varphi)$ for any monotonic $\varphi$ in the sense of the Lebesgue measure and the Hausdorff dimension.

\begin{thm}[Tan-Tian-Wang, \cite{TTW22}]
Let $\varphi\colon\N\to \R_+$ be a non-decreasing function. Then
$$ \mathcal{L}(\mathcal{F}_1(\varphi))=\left\{\begin{array}{ll}
    0, &\text{if}~\sum\limits_{k=1}^\infty \frac{n}{\varphi(n)^2} < \infty;\\
    1, &\text{if}~\sum\limits_{k=1}^\infty \frac{n}{\varphi(n)^2} =\infty.
    \end{array}
  \right.$$

Write
$$\log B=\liminf_{n\to\infty}\frac{\log\varphi(n)}n, ~~ \log b=\liminf_{n\to\infty}\frac{\log \log\varphi(n)}{n}.$$
\begin{itemize}
\item [\rm(1)] When $B=1,$ $\dim_{H}\mathcal{F}_1 (\varphi)=1.$

\item [\rm(2)] When $B=\infty,$ $\dim_{H}\mathcal{F}_1 (\varphi)=1/(1+b).$

\item [\rm(3)] When $1\le B<\infty$,
$$\dim_H{\mathcal{F}_1(\varphi)  }=\inf\{ s\geq0\colon \mathsf{P}(T,-(3s-1)\log B - s\log |T'|)\le0\}.$$
    \end{itemize}
\end{thm}

Motivated by the study of uniform Diophantine approximation,  Kleinbock and Wadleigh \cite{KW18} introduced the following  set of $\varphi$-Dirichlet non-improvable numbers
$$\mathcal{E}_2(\varphi):=\{x\in[0,1)\colon a_{n}(x)a_{n+1}(x)\ge \varphi(n) \text{ for infinitely many } n\in \mathbb{N}\}.$$
They proved that $\mathcal{L}\big(\mathcal{E}_2(\varphi)\big)$ is either null or full according as the series $\sum_{n=1}^{\infty} \log \varphi(n)/\varphi(n)$ converges or diverges respectively.
The Hausdorff dimension of this set (in a generalised form) was established in \cite{HWX20,HKWW18}.

\begin{thm}[Huang-Wu-Xu, \cite{HWX20}]\label{HWX}
Let $\varphi\colon \mathbb{N}\to \mathbb{R}$ be a  positive function. Write
$$\log B=\liminf_{n\to\infty}\frac{\log \varphi(n)}{n} \text{ and } \log b=\liminf_{n\to\infty}\frac{\log\log \varphi(n)}{n}.$$
\begin{itemize}
\item [\rm(1)] When $B=1,$ $\dim_{H}\mathcal{E}_2(\varphi)=1.$

\item [\rm(2)] When $B=\infty,$ $\dim_{H}\mathcal{E}_2(\varphi)=1/(1+b).$

\item [\rm(3)] When $1<B<\infty,$ $$\dim_{H}\mathcal{E}_2(\varphi)=\inf\big\{s\ge0\colon \mathsf{P}(T, -s^{2}\log B-s\log |T'|)\le0\big\}.$$
\end{itemize}
\end{thm}
It is thus clear that the study of the growth rate of the product of two consecutive partial quotients is pivotal
in providing information on the set of Dirichlet non-improvable numbers. For more related results, one can
refer to \cite{BBH,BBH20,HHY21,KK22}.

Naturally we continue the study by considering the set
\begin{align*}
\mathcal{F}_2(\varphi)=\Big\{x\in[0,1)\colon  \exists 1\le k\neq l \le n,  ~&a_{k}(x)a_{k+1}(x)\ge \varphi(n), \\&a_{l}(x)a_{l+1}(x)\ge \varphi(n) \text{ for infinitely many } n\in \N  \Big\}.
\end{align*}
We prove the following dichotomy statement for the Lebesgue measure of $\mathcal{F}_2(\varphi)$ for a general function $\varphi.$
\begin{thm}\label{thm1}
Let $\varphi\colon \mathbb{N}\to \mathbb{R}_{+}$ be a non-decreasing function. Then
\begin{equation*}
         \mathcal{L}\big(\mathcal{F}_2(\varphi)\big)
       =\left\{\begin{array}{ll}  0, & \ \ \ \text{if }~ \sum_{n=1}^{\infty}\frac{\varphi(n)+n\log^{2} \varphi(n)}{\varphi(n)^{2}}<\infty;\\
                                  1, & \ \ \ \text{if }~ \sum_{n=1}^{\infty}\frac{\varphi(n)+n\log^{2} \varphi(n)}{\varphi(n)^{2}}=\infty.
       \end{array}
      \right.
     \end{equation*}
\end{thm}

We also provide a full description on the Hausdorff dimension of $\mathcal{F}_2(\varphi).$
\begin{thm} \label{thm2}
Let $\varphi\colon \mathbb{N}\to \mathbb{R}_{+}$ be a positive function. Write
$$\log B=\liminf_{n\to\infty}\frac{\log \varphi(n)}{n} \text{ and } \log b=\liminf_{n\to\infty}\frac{\log\log \varphi(n)}{n}.$$
\begin{itemize}
\item [\rm(1)] When $B=1,$ $\dim_{H}\mathcal{F}_2(\varphi)=1.$

\item [\rm(2)] When $B=\infty,$ $\dim_{H}\mathcal{F}_2(\varphi)=1/(1+b).$

\item [\rm(3)] When $1<B<\infty,$ $$\dim_{H}\mathcal{F}_2(\varphi)=s_B:=\inf\Big\{s\ge0\colon \mathsf{P}(T, -(3s-1-s^2)\log B-s\log |T'|)\le0\Big\}.$$
\end{itemize}
\end{thm}

Let us make the following remarks regarding Theorems \ref{thm1} and \ref{thm2}:
\begin{itemize}
\item Borel-Bernstein theorem yields that almost all $x$ has unbound partial quotients, which implies that if $\lim_{n\to\infty}\varphi(n)<\infty,$
then $\mathcal{F}_2(\varphi)$ is of  full Lebesgue measure. Therefore we may always assume  that  $\lim_{n\to\infty}\varphi(n)=\infty.$

\item The another key observation is that the monotonicity of $\varphi$ implies that
\begin{align*}
\ \ \ \ \ \ \ \ \  \mathcal{F}_2(\varphi)=\Big\{x\in[0,1)\colon  \exists 1\le k< n&,  ~a_{k}(x)a_{k+1}(x)\ge \varphi(n), \\&a_{n}(x)a_{n+1}(x)\ge \varphi(n) \text{ for infinitely many } n\in \N  \Big\}.
\end{align*}
It is clear that the set on the right-hand side is included in $\mathcal{F}_2(\varphi).$
For the reverse inclusion, since for $x\in \mathcal{F}_2(\varphi),$ there exists an infinite sequence $\{n_t\}_{t\ge 1}$ such that for all $t\ge 1,$
$$ \ \ \ \ \  \ \ \ \ \ \ \ \ \ \ \exists k_t, l_t \text{ with } 1\le k_t<l_t\le n_t \text{ such that } a_{k_t}(x)a_{k_t+1}(x)\ge \varphi(n_t), a_{l_t}(x)a_{l_t+1}(x)\ge \varphi(n_t).$$
Furthermore, $\{l_t\}_{t\ge 1}$ is an infinite set since $\lim_{n}\varphi(n)=\infty$, and thus, by the monotonicity of $\varphi,$
$$a_{k_t}(x)a_{k_t+1}(x)\ge \varphi(l_t), a_{l_t}(x)a_{l_t+1}(x)\ge \varphi(l_t)$$
for infinitely many $l_t$, that is,  $x$ belongs to the set on the right-hand side.

\item Note that $\mathcal{E}_1(\varphi)\subset \mathcal{F}_2(\varphi)\subset\mathcal{E}_2(\varphi),$
and thus the cases $B=1$ and $B=\infty$ in Theorem \ref{thm2} can be directly deduced by Theorems \ref{WW} and \ref{HWX}.

\item One may wonder whether  the results  above and the  methods for proofs  extend  to the case of products of the $m$
 consecutive partial quotients. We believe that this is the case, while the calculations will be complicated and  lengthy.

\end{itemize}

{\noindent \bf  Notation$\colon$}
We use the notation $\xi\ll\eta$ (or $\eta\gg\xi$) if there is an unspecified constant $c$ such that
$\xi\le c\eta$. We write $\xi\asymp\eta$ if $\xi\ll\eta$  and $\xi\gg\eta.$

\section{Preliminaries}
In this section, we recall some fundamental results and concepts for later use.

\subsection{Continued fraction}
For any irrational number $x\in[0,1)$ with continued fraction expansion (\ref{cf}),
the sequences $\{p_n(x)\}_{n\ge0},$ $\{q_{n}(x)\}_{n\ge0}$ satisfy the following recursive relations \cite{KH63}$\colon$
\begin{align*}
 p_{n+1}(x)=a_{n+1}(x)p_{n}(x)+p_{n-1}(x),~
 q_{n+1}(x)=a_{n+1}(x)q_{n}(x)+q_{n-1}(x),
\end{align*}
with the conventions that
 $\big(p_{-1}(x), q_{-1}(x)\big)=\big(1,0\big),~ \big(p_{0}(x), q_{0}(x)\big)=\big(0, 1\big)$.
Since $q_{n}(x)$ is determined by $a_{1}(x),\ldots,a_{n}(x),$  we also write
$q_{n}(a_{1}(x),\ldots,a_{n}(x))$ instead of $q_n(x)$.
We write $a_{n}$ and $q_n$ in place of $a_{n}(x)$ and $q_n(x)$ for simplicity when no confusion can arise.

From the recursive relations, we deduce that $q_n\ge 2^{\frac{n-1}{2}}$ for any $n\ge1$, and
$$1\le \frac{q_{n+m}(a_1,\ldots, a_n,b_1,\ldots, b_m)}{q_{n}(a_1,\ldots, a_n)q_{m}(b_1,\ldots, b_m)}\le 2.$$
Moreover,
$$\prod_{i=1}^{n}a_i\le q_n\le \prod_{i=1}^{n}(a_i+1)\le 2^n\prod_{i=1}^{n}a_i.$$

For  $(a_1,\ldots,a_n)\in \N^{n}$ with $n\in \N,$ we call
$$I_n(a_1,\ldots, a_n)=\{x\in[0,1) \colon a_1(x)=a_1, \ldots , a_n(x)=a_n\}$$
an $n$-th basic  cylinder, which is the set of points whose continued fractions begin with $(a_1, \ldots, a_n)$.
The cylinder takes the form
\begin{equation*}
I_n(a_1,\ldots, a_n)=
	\begin{cases}
		\Big[\frac{p_n}{q_n},\frac{p_n+p_{n-1}}{q_n+q_{n-1}}\Big),\quad  & \text{if $n$ is even};\\
		\Big(\frac{p_n+p_{n-1}}{q_n+q_{n-1}},\frac{p_n}{q_n}\Big],\quad  & \text{if $n$ is odd}.
	\end{cases}
\end{equation*}
And its length satisfies
$$\frac{1}{2q^2_n}\le |I_n(a_1,\ldots, a_n)|=\frac{1}{q_n(q_n+q_{n-1})}\le \frac{1}{q^2_n},$$
here and hereafter, $|\cdot|$ stands for the length of an interval. We then see that
$$1\le\sum_{a_1,\ldots,a_{n}\in\N }\frac{1}{q^{2}_{n}}\le 2.$$
The basic cylinder of order $n$ which contains $x$ will be denoted by  $I_{n}(x)$, i.e.,
       $I_{n}(x)=I_{n}(a_{1}(x),\ldots,a_{n}(x))$.

The following proposition describes the  position of the basic cylinders.
\begin{pro}[Khintchine, \cite{KH63}] \label{cylinder}
Let $I_{n}=I_n(a_1,\ldots, a_n)$ be a cylinder of order $n,$ which is partitioned into sub-cylinders
$\{I_{n+1}(a_1,\ldots,a_n,a_{n+1})\colon a_{n+1}\in \N\}.$
When $n$ is odd, these sub-cylinders are positioned from left to right, as $a_{n+1}$ increases from $1$ to $\infty;$
when $n$ is even, they are positioned from right to left.
\end{pro}

For the forthcoming calculations, we introduce some sets in terms of the products of partial quotients. We fix $l\ge 1$. For $1\le k \le n$, we define
  $$J_k=\bigcup_{a_ka_{k+1}\ge l}I_{n+1}(a_1,\ldots,a_n,a_{n+1}),$$
where the union is taken over all pairs $(a_k,a_{k+1})$ of partial quotients satisfying $a_ka_{k+1}\ge l$, while all other partial quotients $a_1,\ldots,a_{k-1}, a_{k+2},\ldots, a_{n+1}$ are fixed.
In a similar way, we define
\begin{align*}
  \widetilde{J}_k&=\bigcup_{1\le a_ka_{k+1}< l}I_{n+1}(a_1,\ldots,a_n,a_{n+1});\\
  H_{n}&=\bigcup_{a_{n-1}a_n\ge l,~a_na_{n+1}\ge l}I_{n+1}(a_1,\ldots,a_n,a_{n+1}),\\
  \widetilde{H}_{n}&=\bigcup_{1\le a_{n-1}a_n<l,~a_na_{n+1}\ge l}I_{n+1}(a_1,\ldots,a_n,a_{n+1}).
\end{align*}

The following lemma concerns the Lebesgue measures of these sets.
\begin{lem}\label{lem2.1}
With the notations above, we have that
\begin{itemize}
\item [\rm(1)] $\mathcal L(J_k)=\frac{\log l+O(1)}{l}|I_{n-1}(a_1,\ldots,a_{k-1}, a_{k+2},\ldots, a_{n+1})|$.

    (And thus  $\mathcal L(\widetilde{J}_k)=(1-\frac{\log l+O(1)}{l})|I_{n-1}(a_1,\ldots,a_{k-1}, a_{k+2},\ldots, a_{n+1})|$.)
\smallskip

\item [\rm(2)]  $\mathcal L(H_n)\asymp\frac1l |I_{n-2}(a_1,a_2,\ldots,a_{n-2})|$ and $\mathcal L(\widetilde{H}_{n})\asymp\frac{\log l}{l}|I_{n-2}(a_1,a_2,\ldots,a_{n-2})|.$
\end{itemize}
\end{lem}
\begin{proof}
Item (1) can be obtained by Lemma 3.1 in \cite{HHY21}. As for (2), we write
$$H_n=\bigcup_{a_{n-1}\in \mathbb{N}}\bigcup_{a_{n}>l}\bigcup_{a_{n+1}\in \mathbb{N}}I_{n+1}(a_1,\ldots,a_n,a_{n+1})
     \cup \bigcup_{1\le a_n< l} \bigcup_{a_{n-1}\ge \frac{l}{a_n}}\bigcup_{a_{n+1}\ge\frac{l}{a_n}}I_{n+1}(a_1,\ldots,a_n,a_{n+1}).$$
Then we have
\begin{align*}
\mathcal L(H_n)&\asymp\sum_{a_{n-1}=1}^{\infty}\sum_{a_n>l}\sum_{a_{n+1}=1}^{\infty}|I_{n+1}(a_1,\ldots,a_{n+1})|
              +\sum_{1\le a_n< l}\sum_{a_{n-1}\ge \frac{l}{a_n}}\sum_{a_{n+1}\ge\frac{l}{a_n}}    |I_{n+1}(a_1,\ldots,a_{n+1})|\\
    &\asymp|I_{n-2}(a_1,\ldots,a_{n-2})|\Big(\sum_{a_n>l}\frac{1}{a_{n}^{2}}+\sum_{1\le a_n< l}\sum_{a_{n-1}\ge \frac{l}{a_n}}\sum_{a_{n+1}\ge\frac{l}{a_n}}\frac{1}{a_{n+1}^{2}}\cdot\frac{1}{a_{n-1}^{2}}\cdot\frac{1}{a_{n}^{2}}\Big)\\
    &\asymp |I_{n-2}(a_1,\ldots,a_{n-2})|\Big(\frac{1}{l}+\sum_{1\le a_n< l}\frac{1}{a_{n}^{2}}\cdot\frac{a_n}{l}\cdot\frac{a_n}{l}\Big)\asymp\frac{1}{l}|I_{n-2}(a_1,\ldots,a_{n-2})|.
\end{align*}
Similarly, since
$$\widetilde{H}_{n}=\bigcup_{1\le a_n< l} \bigcup_{a_{n-1}< \frac{l}{a_n}}\bigcup_{a_{n+1}\ge\frac{l}{a_n}}I_{n+1}(a_1,\ldots,a_n,a_{n+1}),$$
we deduce that
\begin{align*}
\mathcal L(\widetilde{H}_{n})&\asymp
              \sum_{1\le a_n< l}\sum_{a_{n-1}< \frac{l}{a_n}}\sum_{a_{n+1}\ge\frac{l}{a_n}}|I_{n+1}(a_1,\ldots,a_n,a_{n+1})|\\
    &\asymp|I_{n-2}(a_1,\ldots,a_{n-2})|\sum_{1\le a_n<l}\sum_{a_{n-1}< \frac{l}{a_n}}\sum_{a_{n+1}\ge\frac{l}{a_n}}\frac{1}{a_{n+1}^{2}}\cdot\frac{1}{a_{n-1}^{2}}\cdot\frac{1}{a_{n}^{2}}\\
    &\asymp |I_{n-2}(a_1,\ldots,a_{n-2})|\sum_{1\le a_n<l}\Big(\frac{1}{la_n}-\frac{1}{l^{2}}\Big)
   \asymp\frac{\log l}{l}|I_{n-2}(a_1,\ldots,a_{n-2})|.
\end{align*}
\end{proof}

\subsection{Chung-Erd\"{o}s inequality} The following lemmas are widely used to estimate the measure of a limsup set in a probability space.

\begin{lem}[Borel-Cantelli lemma] \label{BC lemma} Let $(\Omega, \mathcal{B}, \nu)$ be a finite measure space. Assume that $\{A_n\}_{n\ge 1}$ is a sequence of measurable sets,
 $A=\limsup A_n=\bigcap_{N=1}^{\infty}\bigcup_{n=N}^{\infty}A_n.$   Then
\begin{align*}
  \nu(\limsup A_n)=\left\{
         \begin{array}{ll}
           0, & \hbox{if $\sum_{n= 1}^{\infty}\nu(A_n)<\infty$;} \\
           \nu(\Omega), & \hbox{if $\sum_{n= 1}^{\infty}\nu(A_n)=\infty$~\text{and} $\{A_n\}_{n\ge 1}$\text{ are pairwise independent}.}
         \end{array}
       \right.
\end{align*}
\end{lem}
We mention that in the divergence part of the above lemma the independence property is needed.
While in many applications the sets $\{A_n\}_{n\ge 1}$ are not pairwise independent, and one uses the following lemma instead for the divergence part.
\begin{lem}[Chung-Erd\"{o}s inequality, \cite{Chung}]
If $\sum_{n\ge 1}\nu(A_n)=\infty$, then $$
\nu(A)\ge \limsup_{N\to \infty}\frac{(\sum_{1\le n\le N}\nu(A_n))^2}{\sum_{1\le i\ne j\le N}\nu(A_i\cap A_j)}.
$$
\end{lem}

In many cases, Chung-Erd\"{o}s inequality only enables us to conclude the positiveness of $\nu(A)$. To get a full measure result for $A$, one may apply Chung-Erd\"{o}s inequality locally,
and then arrives at the full measure of $A$ in the light of Knopp's lemma.

\begin{lem}[Knopp, \cite{Knopp}]
 If $A\subset [0,1)$ is a Lebesgue measurable set and $\mathcal{C}$ is a class of subintervals of $[0,1)$ satisfying
\begin{itemize}
\item [\rm(1)] every open subinterval of $[0,1)$ is at most a countable union of disjoint elements
from $\mathcal{C}$,

\item [\rm(2)] for any $B\in \mathcal{C}, \mathcal{L}(A\cap B)\ge c\mathcal{L}(B),$ where $c$ is a constant independent
of $B$,

\end{itemize}
then $\mathcal{L}(A)=1.$
\end{lem}

\subsection{Pressure function}\label{pressure}
Pressure function is an appropriate tool in dealing with dimension problems in infinite conformal iterated function systems.
We recall a fact that the pressure function with a continuous potential can be approximated by the pressure functions restricted to the sub-systems in continued fractions.
For more information on pressure functions, we refer the readers to \cite{MU96,MU99,MU03}.

Let  $\mathcal{A}$ be a finite or infinite subset of $\N$. Define
$$X_{\mathcal{A}}=\big\{x\in[0,1)\colon a_{n}(x)\in \mathcal{A} \text{ for all } n\geq 1\big\}.$$
Then, with the Gauss map $T$ restricted to it, $X_{\mathcal{A}}$ forms a dynamical system.
The pressure function restricted to this subsystem $(X_{\mathcal{A}},T)$ with potential $\phi\colon [0,1)\to \mathbb{R} $ is defined as
\begin{equation}\label{sub}
\mathsf{P}_{\mathcal{A}}(T,\phi)=\lim_{n\to\infty}\frac{\log\sum\limits_{(a_{1},\ldots,a_{n})\in \mathcal{A}^{n}}\sup\limits_{x\in X_{\mathcal{A}}}\exp{S_n \phi([a_{1},\ldots,a_{n}+x])}}{n},
\end{equation}
where $S_{n}\phi(x)$ denotes the ergodic sum $\phi(x)+\cdots+\phi(T^{n-1}(x))$.
When $\mathcal{A}=\mathbb{N}$, we write $\mathsf{P}(T,\phi)$ for $\mathsf{P}_{\mathbb{N}}(T,\phi)$.

The $n$-th variation $\textrm{Var}_{n}(\phi)$ of $\phi$ is defined as
$$\textrm{Var}_{n}(\phi):=\sup\big\{|\phi(x)-\phi(y)|\colon I_{n}(x)=I_{n}(y)\big\}.$$
The existence of the limit in equation (\ref{sub}) is due to the following result.

\begin{lem} [Walters, \cite{PW}]
The limit defining $\mathsf{P}_{\mathcal{A}}(T,\phi)$  in (\ref{sub}) exists. Moreover,
if $\phi\colon [0,1)\rightarrow \mathbb{R}$ satisfies $\textrm{Var}_{1}(\phi)<\infty$ and $\textrm{Var}_{n}(\phi)\rightarrow 0$ as $n\rightarrow \infty$,
the value of $\mathsf{P}_{\mathcal{A}}(T,\phi)$ remains the same even without taking the supremum over $x\in X_{\mathcal{A}}$ in (\ref{sub}).
\end{lem}

The following lemma presents a continuity property of the pressure function when the continued fraction system $([0,1), T)$ is approximated by its subsystems $(X_{\mathcal{A}},T).$
For the detailed proof, see \cite{HMU02} or \cite{LWWX}.

\begin{lem} [Hanus-Mauldin-Urba\'{n}ski, \cite{HMU02}] \label{approximation}
Let $\phi\colon [0,1)\to \mathbb{R}$ be a real function with $\text{Var}_1(\phi)<\infty$ and $\text{Var}_n(\phi)\to 0$ as $n\to\infty.$ We have
$$\mathsf{P}(T,\phi)=\sup\{\mathsf{P}_{\mathcal{A}}(T,\phi)\colon \mathcal{A} \text{ is a finite subset of }\N\}.$$
\end{lem}

Now we consider a specific potential
$$\phi_1(x)=-(3s-1-s^2)\log B-s\log|T'(x)|,$$
where $1<B<\infty$ and $s\ge 0.$ Clearly, $\phi_1$ satisfies the variation condition and   Lemma \ref{approximation} holds. The pressure function (\ref{sub}) with potential $\phi_1$ reads
$$\mathsf{P}(T,\phi_1) =\lim_{n\to\infty}\frac{\log\sum\limits_{(a_{1},\ldots,a_{n})\in \mathcal{A}^{n}}\exp{S_n\phi_1(x)}}{n}\\
                       =\lim_{n\to\infty}\frac{\log\sum\limits_{(a_{1},\ldots,a_{n})\in \mathcal{A}^{n}}B^{-n(3s-1-s^2)}q_n^{-2s}}{n}.$$

For $n\ge 1$ and $s\ge 0,$ we put
$$f_n(s)=\sum_{(a_1,\ldots,a_n)\in \mathcal{A}^{n}}B^{-n(3s-1-s^2)}q_n^{-2s},$$
and
\begin{align*}
  &s_{n,B}(\mathcal{A})=\inf\{s\ge 0\colon f_n(s)\le 1\},\\
  &s_{B}(\mathcal{A})=\inf\big\{s\ge0\colon \mathsf{P}_{\mathcal{A}}(T, -(3s-1-s^2)\log B-s\log|T'|)\le 0\big\},\\
  &s_{B}(\N)=\inf\big\{s\ge0\colon \mathsf{P}(T, -(3s-1-s^2)\log B-s\log|T'|)\le 0\big\}.
\end{align*}

If $\mathcal{A}\subset \N$ is a finite set, it is readily checked that both $f_n(s)$ and $\mathsf{P}_{\mathcal{A}}(T,\phi_1)$
are continuous and monotonically decreasing  with respect to $s.$
Therefore, $s_{n,B}(\mathcal{A})$ and $s_{B}(\mathcal{A})$ are respectively the unique solutions of $f_n(s)=1$ and $\mathsf{P}_{\mathcal{A}}(T, \phi_1)=0.$

For $M\in \N,$ put $\mathcal{A}_M=\{1,2,\ldots,M\}.$ For simplicity, we write $s_{n,B}(M)$ for $s_{n,B}(\mathcal{A}_M),$
$s_{B}(M)$ for $s_{B}(\mathcal{A}_M),$ $s_{n,B}$ for $s_{n,B}(\N)$ and $s_B$ for $s_B(\N).$

\begin{cor}\label{corollary}
For any integer $M\in \N,$
$$\lim_{n\to\infty}s_{n,B}(M)=s_{B}(M),\ \ \lim_{M\to\infty}s_{B}(M)=s_B.$$
Noting that the dimensional number $s_B$ is continuous with respect to $B\in (1,\infty),$ we obtain
$$\lim_{B\to 1}s_B=1, \ \ \lim_{B\to\infty} s_B=\frac{1}{2}.$$
\end{cor}
\begin{proof}
The similar arguments as Lemmas 2.4-2.6 in \cite{WW08} apply.
\end{proof}

\section{Proof of Theorem \ref{thm1}}
The proof of Theorem \ref{thm1} consists of two parts: the convergence part and the divergence part. We first recall that
\begin{align*}
\mathcal{F}_2(\varphi)=\Big\{x\in[0,1)\colon  \exists 1\le k< n,  ~&a_{k}(x)a_{k+1}(x)\ge \varphi(n), \\&a_{n}(x)a_{n+1}(x)\ge \varphi(n) \text{ for infinitely many } n\in \N  \Big\}.
\end{align*}
Write
$$A_n=\big\{x\in[0,1)\colon \exists 1\le k< n,  a_{k}(x)a_{k+1}(x)\ge \varphi(n), a_{n}(x)a_{n+1}(x)\ge\varphi(n)\big\}.$$
Then the set $\mathcal{F}_2(\varphi)$ can further be written as
$$\mathcal{F}_2(\varphi)=\bigcap_{N=1}^{\infty}\bigcup_{n=N}^{\infty}A_n.$$

\subsection{The convergence part of Theorem \ref{thm1}}
In this subsection, we prove that
$$\mathcal{L}\big(\mathcal{F}_2(\varphi)\big)=0 \text{ when }\sum_{n=1}^{\infty}\frac{\varphi(n)+n\log^{2} \varphi(n)}{\varphi(n)^{2}}<\infty.$$
As usual, we will use the convergence part of the Borel-Cantelli Lemma \ref{BC lemma}. The main task is to estimate the Lebesgue measure of the set $A_{n}.$
\begin{lem}
The Lebesgue measure of $A_n$ satisfies
$$\mathcal{L}(A_n)\ll\frac{n\log^{2} \varphi(n)}{\varphi(n)^{2}}+\frac{1}{\varphi(n)}.$$
\end{lem}
\begin{proof}
The set $A_n$ can be written as a union over a collection of $(n+1)$-order cylinders:
$$A_{n}=\bigcup_{k=1}^{n-1}\bigcup_{\substack{
                                      a_1,\ldots,a_n,a_{n+1}\in \mathbb{N}\colon \\
                                      a_ka_{k+1}\ge \varphi(n), a_{n}a_{n+1}\ge \varphi(n)}}I_{n+1}(a_1,\ldots,a_{n},a_{n+1}).$$
By Lemma \ref{lem2.1}, we have
\begin{align*}
  \mathcal{L}(A_n)&\le\sum_{k=1}^{n-2}\sum_{\substack{
                                      (a_1,\ldots,a_{n-1})\in \mathbb{N}^{n-1}\colon \\
                                     a_ka_{k+1}\ge \varphi(n)}} \sum_{a_{n}a_{n+1}\ge \varphi(n)}
                                    |I_{n+1}| +\sum_{\substack{
                                      (a_1,\ldots,a_n,a_{n+1})\in \mathbb{N}^{n+1}\colon \\
                                     a_{n-1}a_{n}\ge \varphi(n), a_{n}a_{n+1}\ge \varphi(n)}} |I_{n+1}|& \\
                &\asymp\sum_{k=1}^{n-2}\sum_{\substack{
                                      (a_1,\ldots,a_{n-1})\in \mathbb{N}^{n-1}\colon\\
                                     a_ka_{k+1}\ge \varphi(n)}} |I_{n-1}|\cdot\frac{\log \varphi(n)}{\varphi(n)}
                                    +\frac{1}{\varphi(n)}\\
                &\asymp\frac{n\log^{2} \varphi(n)}{\varphi(n)^{2}}+\frac{1}{\varphi(n)}.
\end{align*}
\end{proof}

\subsection{The divergence part of Theorem \ref{thm1}}
In this subsection, we will prove that
$$\mathcal{L}\big(\mathcal{F}_2(\varphi)\big)=1 \text{ when }\sum_{n=1}^{\infty}\frac{\varphi(n)+n\log^{2} \varphi(n)}{\varphi(n)^{2}}=\infty.$$
As mentioned, our strategy is to apply the Chung-Erd\"os inequality in each basic cylinder, and then to use Knopp's lemma.

Let us start by treating the condition that
$$\sum_{n=1}^{\infty}\frac{\varphi(n)+n\log^{2} \varphi(n)}{\varphi(n)^{2}}=\infty.$$
We  may assume without loss of generality  that $\varphi(n)\ge n\log \varphi(n)$ for all $n\ge 1.$ Otherwise, we introduce an auxiliary function
$$\phi(n):=\max\{\varphi(n), n\log \varphi(n)\}.$$
Obviously, $\phi\ge\varphi$ and $\phi$ is a non-decreasing function, and thus $\mathcal{F}_2(\phi)$ is a subset of $\mathcal{F}_2(\varphi).$ We now check that
$$\sum_{n=1}^{\infty}\frac{\varphi(n)+n\log^{2} \varphi(n)}{\varphi(n)^{2}}=\infty
\Rightarrow
\sum_{n=1}^{\infty}\frac{\phi(n)+n\log^{2} \phi(n)}{\phi(n)^{2}}=\infty.$$
Since    $\varphi(n)<n\log \varphi(n)$ infinitely often,
we  choose an infinite subsequence   $\{m_k\}_{k\ge 1}$  with
$$m_{k+1}>2m_k \text { and } \varphi(m_k)<m_k\log \varphi(m_k).$$
It follows that
\begin{align*}
  \sum_{n=1}^{\infty}\frac{\phi(n)+n\log^{2} \phi(n)}{\phi(n)^{2}}
  &\ge \sum_{k=1}^{\infty}\sum_{\frac{m_{k+1}}{2}\le n\le m_{k+1}}\frac{n}{(\phi(n)/\log \phi(n))^{2}}\\
  &\ge\sum_{k=1}^{\infty}\sum_{\frac{m_{k+1}}{2}\le n\le m_{k+1}}\frac{m_{k+1}/2}{m_{k+1}^{2}}
  \ge\sum_{k=1}^{\infty}\frac{1}{4}=\infty.
\end{align*}

Throughout this subsection, we fix a basic $t$-th cylinder $I_{t}=I_t(b_1,\ldots,b_t)$.
To utilise the Chung-Erd\"{o}s inequality on $\mathcal{F}_2(\varphi)\cap I_t=\limsup (A_n\cap I_t)$,
we need estimate a lower bound for the Lebesgue measure of $A_n\cap I_t$ and an upper bound for the Lebesgue measure of $A_n\cap A_m\cap I_t$.

\begin{lem}\label{measure1}
For $n$ large enough,  we have
$$\mathcal{L}(A_n\cap I_t)\gg \frac{\varphi(n)+n\log^{2} \varphi(n)}{\varphi(n)^{2}}\cdot |I_t|.$$
\end{lem}

\begin{proof}
Since $\varphi(n)\to\infty$, there exists  $N\in \mathbb{N}$ such that, for  all $n\ge N,$
 $n\ge 2t+4$, $\varphi(n)>\Big(\max\{b_1,\ldots,b_t\}\Big)^{2}$, and thus
$$A_n\cap I_t=\big\{x\in I_t(b_1,\ldots,b_t)\colon \exists t< k< n,  a_{k}(x)a_{k+1}(x)\ge \varphi(n), a_{n}(x)a_{n+1}(x)\ge\varphi(n)\big\}.$$
In order to prove the lemma, we construct a subset $\mathcal{B}_{n}(\varphi)$ of $A_n\cap I_t,$ and show that
$$\mathcal{L}(\mathcal{B}_{n}(\varphi))\gg \frac{\varphi(n)+n\log^{2} \varphi(n)}{\varphi(n)^{2}}\cdot |I_t|.$$

Assume that $n-t$ is even (the same conclusion can be drawn for the case $n-t$ is odd). Set
$$\mathcal{B}_{n-1,n}(\varphi):=\{x\in I_t(b_1,\ldots,b_t)\colon a_{n-1}(x)a_{n}(x)\ge \varphi(n), a_n(x)a_{n+1}(x)\ge\varphi(n)\}.$$
For $t<i<n-1$ with $n-i$ odd (or $i-t$ odd), define $\mathcal{B}_{i,n}(\varphi)$ as
\begin{align*}
\mathcal{B}_{i,n}(\varphi):=\Big\{x\in I_t(b_1,\ldots,b_t)\colon &a_{i}(x)a_{i+1}(x)\ge \varphi(n),~a_n(x)a_{n+1}(x)\ge\varphi(n),
                                                              \\ &a_{n-(2k+1)}(x)a_{n-2k}(x)< \varphi(n) \text{ for } 0\le k<\frac{n-i-1}{2} \Big\}.
\end{align*}
Obviously we have that  $\mathcal{B}_{i,n}(\varphi)\subset A_{n}\cap I_t$ and $\mathcal{B}_{i,n}(\varphi)\cap\mathcal{B}_{j,n}(\varphi)=\emptyset$ for $t<i\neq j<n.$
We then define the desirable subset by
$$\mathcal{B}_n(\varphi)=\mathcal{B}_{n-1,n}(\varphi)\cup\bigcup_{t<i<n-1, ~n-i \text{ is odd}}\mathcal{B}_{i,n}(\varphi).$$

Now we devote to estimating the Lebesgue measure of $\mathcal{B}_n(\varphi).$  By Lemma \ref{lem2.1}, we have
$$\mathcal{L}(\mathcal{B}_{n-1,n}(\varphi))\asymp\frac{1}{\varphi(n)}|I_t|.$$
For $t<i<n-1,$ we write
\begin{align*}
  \mathcal{C}_{i,n}(\varphi):=\bigcup_{a_ia_{i+1}\ge \varphi(n)}\
                                \bigcup_{\substack{a_{n-(2k+1)} a_{n-2k}<\varphi(n)\\ 0< k<\frac{n-i-1}{2}}}
                                 \bigcup_{\substack{1\le a_{n-1}a_n<\varphi(n)\\ a_na_{n+1}\ge \varphi(n)}}  I_{n+1}(b_1,\ldots,b_t,\ldots,a_{i},\ldots,a_{n+1}).
\end{align*}

Applying Lemma \ref{lem2.1} again, we deduce that
\begin{align*}
  \mathcal{L}(\mathcal{C}_{i,n}(\varphi))&=\sum_{a_ia_{i+1}\ge \varphi(n)}\sum_{\substack{a_{n-(2k+1)} a_{n-2k}<\varphi(n)\\
                                             0< k<\frac{n-i-1}{2}}}\sum_{\substack{
                                                                  1\le a_{n-1}a_n<\varphi(n)\\
                                                                  a_na_{n+1}\ge \varphi(n)}}|I_{n+1}|\\
                                         &\gg\sum_{a_ia_{i+1}\ge \varphi(n)}\sum_{\substack{
                                            a_{n-(2k+1)} a_{n-2k}<\varphi(n)\\
                                             0< k<\frac{n-i-1}{2}}}|I_{n-2}|\cdot \frac{\log \varphi(n)}{\varphi(n)}\\
                                         &\gg\frac{\log \varphi(n)+O(1)}{\varphi(n)}\cdot\Big(1-\frac{\log \varphi(n)+O(1)}{\varphi(n)}\Big)^{n-i}\cdot|I_{i-1}|\cdot\frac{\log \varphi(n)}{\varphi(n)}\\
                                         &\asymp \frac{\log^{2} \varphi(n)}{\varphi(n)^{2}}\cdot
                                          \Big(1-\frac{\log \varphi(n)}{2\varphi(n)}\Big)^{n-i}\cdot|I_{i-1}|\ge\frac{\log^{2} \varphi(n)}{\varphi(n)^{2}}
                                          \Big(1-\frac{1}{2n}\Big)^{n-i}|I_{i-1}|,
                                           \end{align*}
where in the last inequality, we used $\varphi(n)\ge n\log \varphi(n)$ as assumed.

We readily check  that
$$\bigcup_{\substack{t<i<n-1\\i-t \text{~odd}}}\mathcal{B}_{i,n}(\varphi)\supset \mathcal{C}_{t+1,n}(\varphi)
\cup \Big(\bigcup_{\substack{t+1<i<n-1\\ ~n-i \text{~odd}}}\,\bigcup_{a_{t+1},\ldots,a_{i-1}\in \N} \mathcal{C}_{i,n}(\varphi)\Big),$$
and thus
\begin{align*}
  \mathcal{L}\left(\bigcup_{\substack{t<i<n-1\\i-t \text{~ odd}}}\mathcal{B}_{i,n}(\varphi)\right)
  &\ge  \mathcal{L}(\mathcal{C}_{t+1,n}(\varphi))+\sum_{\substack{t+1<i<n-1\\n-i \text{~ odd}}}\sum_{a_{t+1},\ldots,a_{i-1}\in \N}\mathcal{L}(\mathcal{C}_{i,n}(\varphi))\\
  &\gg|I_t|\cdot\frac{\log^{2} \varphi(n)}{\varphi(n)^{2}}\cdot\sum_{\substack{t<i<n-1\\n-i \text{~ odd}}}\Big(1-\frac{1}{2n}\Big)^{n-i}\\
  &\gg\frac{n\log^{2} \varphi(n)}{\varphi(n)^{2}}\cdot|I_t|.
  \end{align*}

Hence we have
\begin{align*}
  \mathcal{L}(A_n\cap I_t)\ge \mathcal{L}(\mathcal{B}_n(\varphi))=\mathcal{L}(\mathcal{B}_{n-1,n}(\varphi))+\sum_{\substack{t<i<n-1\\ n-i \text{~ odd}}}\mathcal{L}(\mathcal{B}_{i,n}(\varphi))
                                                                    \gg |I_t|\Big(\frac{1}{\varphi(n)}+\frac{n\log^{2} \varphi(n)}{\varphi(n)^{2}}\Big),
\end{align*}
which completes the proof.
\end{proof}

\begin{lem}\label{measure2}
For $n, m$ large enough with $n> m+2,$ we have
$$\mathcal{L}(A_n\cap A_m\cap I_t)\ll \frac{6\log^{2}\varphi(n)}{\varphi(n)^{2}}\cdot|I_t|
                                       +\left(\frac{n\log^{2} \varphi(n)}{\varphi(n)^{2}}+\frac{1}{\varphi(n)}\right)\cdot\left(\frac{m\log^{2} \varphi(m)}{\varphi(m)^{2}}+\frac{1}{\varphi(m)}\right)\cdot|I_t|.$$
\end{lem}

\begin{proof}
Note that
\begin{align*}
A_n\cap A_m\cap I_t=\left\{x\in I_t\colon
\begin{array}{l}
   \exists 1\le l< m,~a_{l}(x)a_{l+1}(x)\ge \varphi(m), ~a_{m}(x)a_{m+1}(x)\ge \varphi(m) \\
    \exists 1\le k< n,~a_{k}(x)a_{k+1}(x)\ge \varphi(n), ~a_{n}(x)a_{n+1}(x)\ge \varphi(n)
  \end{array}\right\}.
\end{align*}
Since $\varphi(m)\le \varphi(n),$ the above set is contained in the union of the following three sets which are defined according to the relationship between $k$ and $m\colon$
\begin{itemize}
\item [$\bullet$]   for $k<m,$
$$U_1=\{x\in I_t\colon  \exists k< m,~a_{k}(x)a_{k+1}(x)\ge \varphi(n),~a_{m}(x)a_{m+1}(x)\ge \varphi(m), ~a_{n}(x)a_{n+1}(x)\ge \varphi(n)\};$$
\item [$\bullet$]  for $k=m,$
$$U_2=\{x\in I_t\colon  \exists l< m,~a_{l}(x)a_{l+1}(x)\ge \varphi(m),~a_{m}(x)a_{m+1}(x)\ge \varphi(n),~a_{n}(x)a_{n+1}(x)\ge \varphi(n)\};$$
\item [$\bullet$]  for $k>m,$
$$U_3=\left\{x\in I_t\colon  \exists l<m<k< n,
                                               \begin{array}{c}
                                                 a_{l}(x)a_{l+1}(x)\ge \varphi(m),~a_{m}(x)a_{m+1}(x)\ge \varphi(m), \\
                                                  a_{k}(x)a_{k+1}(x)\ge \varphi(n),~a_{n}(x)a_{n+1}(x)\ge \varphi(n)\\
                                               \end{array}
                                             \right\}.$$
\end{itemize}

Now we estimate the Lebesgue measures of these three sets respectively. By Lemma \ref{lem2.1}, we have that
\begin{align*}
  \mathcal{L}(U_1)\le &\sum_{k=1}^{m-2}\mathcal{L}\{x\in I_t\colon a_k(x)a_{k+1}(x)\ge \varphi(n), a_{m}(x)a_{m+1}(x)\ge \varphi(m), a_n(x)a_{n+1}(x)\ge \varphi(n)\}\\&+\mathcal{L}\{x\in I_t\colon a_{m-1}(x)a_{m}(x)\ge \varphi(n), a_{m}(x)a_{m+1}(x)\ge \varphi(m), a_n(x)a_{n+1}(x)\ge \varphi(n)\}\\ \ll &\frac{m\log \varphi(m)}{\varphi(m)}\cdot\frac{\log^{2}\varphi(n)}{\varphi(n)^{2}}\cdot|I_t|+\frac{\log^{2}\varphi(n)}{\varphi(n)^{2}}\cdot|I_t|
  \le \frac{2\log^{2}\varphi(n)}{\varphi(n)^{2}}\cdot|I_t|.
\end{align*}
Likewise, we deduce that
\begin{align*}
  \mathcal{L}(U_2)\ll& \frac{\log \varphi(n)}{\varphi(n)}\sum_{l=1}^{m-1}\mathcal{L}\{x\in I_t\colon  a_{l}(x)a_{l+1}(x)\ge \varphi(m),~a_{m}(x)a_{m+1}(x)\ge \varphi(n)\}\\
  =&\frac{\log \varphi(n)}{\varphi(n)}\sum_{l=1}^{m-2}\mathcal{L}\{x\in I_t\colon  a_{l}(x)a_{l+1}(x)\ge \varphi(m),~a_{m}(x)a_{m+1}(x)\ge \varphi(n)\}\\
  &+\frac{\log \varphi(n)}{\varphi(n)}\mathcal{L}\{x\in I_t\colon  a_{m-1}(x)a_{m}(x)\ge \varphi(m),~a_{m}(x)a_{m+1}(x)\ge \varphi(n)\}\\
  \ll&m\cdot\frac{\log \varphi(n)}{\varphi(n)}\cdot\frac{\log \varphi(m)}{\varphi(m)}\cdot\frac{\log \varphi(n)}{\varphi(n)}\cdot |I_t|+\frac{\log^{2} \varphi(n)}{\varphi(n)^{2}}\cdot |I_t|\le \frac{2\log^{2}\varphi(n)}{\varphi(n)^{2}}\cdot|I_t|.
\end{align*}
The estimation for $\mathcal{L}(U_3)$ is fulfilled by considering three cases according as that $k$ is equal to $m+1$, or $n-2$, or in between.
Precisely,
\begin{align*}
 \mathcal{L}(U_3)\le&\sum_{l=t+1}^{m-1}\sum_{k=m+1}^{n-1}\mathcal{L}\left\{x\in I_t\colon
                                               \begin{array}{c}
                                                 a_{l}(x)a_{l+1}(x)\ge \varphi(m),~a_{m}(x)a_{m+1}(x)\ge \varphi(m), \\
                                                  a_{k}(x)a_{k+1}(x)\ge \varphi(n),~a_{n}(x)a_{n+1}(x)\ge \varphi(n)\\
                                               \end{array}
                                             \right\}\\
                                             =&\sum_{l=t+1}^{m-1}\sum_{k=m+2}^{n-2}\mathcal{L}\left\{x\in I_t\colon
                                               \begin{array}{c}
                                                 a_{l}(x)a_{l+1}(x)\ge \varphi(m),~a_{m}(x)a_{m+1}(x)\ge \varphi(m), \\
                                                  a_{k}(x)a_{k+1}(x)\ge \varphi(n),~a_{n}(x)a_{n+1}(x)\ge \varphi(n)\\
                                               \end{array}
                                             \right\}\\  &+\sum_{l=t+1}^{m-1}\mathcal{L}\left\{x\in I_t\colon
                                               \begin{array}{c}
                                                 a_{l}(x)a_{l+1}(x)\ge \varphi(m),~a_{m}(x)a_{m+1}(x)\ge \varphi(m), \\
                                                  a_{m+1}(x)a_{m+2}(x)\ge \varphi(n),~a_{n}(x)a_{n+1}(x)\ge \varphi(n)\\
                                               \end{array}
                                             \right\}\\&+\sum_{l=t+1}^{m-1}\mathcal{L}\left\{x\in I_t\colon
                                               \begin{array}{c}
                                                 a_{l}(x)a_{l+1}(x)\ge \varphi(m),~a_{m}(x)a_{m+1}(x)\ge \varphi(m), \\
                                                  a_{n-1}(x)a_{n}(x)\ge \varphi(n),~a_{n}(x)a_{n+1}(x)\ge \varphi(n)\\
                                               \end{array}
                                             \right\}\\
                                             \ll&|I_t|\Big[\frac{n\log^{2} \varphi(n)}{\varphi(n)^{2}}\Big(\frac{m\log^{2}\varphi(m)}{\varphi(m)^{2}}
                                             +\frac{1}{\varphi(m)}\Big) +
                                             \frac{2\log^{2}\varphi(n)}{\varphi(n)^{2}}
                                             +\frac{1}{\varphi(n)}\Big(\frac{m\log^{2}\varphi(m) }{\varphi(m)^{2}}+\frac{1}{\varphi(n)}\Big)\Big]\\
                                             =& \frac{2\log^{2}\varphi(n)}{\varphi(n)^{2}}\cdot|I_t|
                                             +\left(\frac{n\log^{2} \varphi(n)}{\varphi(n)^{2}}+\frac{1}{\varphi(n)}\right)\cdot\left(\frac{m\log^{2} \varphi(m)}{\varphi(m)^{2}}+\frac{1}{\varphi(m)}\right)\cdot|I_t|.
\end{align*}
\end{proof}

Having Lemmas \ref{measure1} and \ref{measure2} in hand,  we finish the proof as follows.
From Lemma \ref{measure1} we have that
$\sum_{n=1}^{\infty}\mathcal{L}(A_n\cap I_t)=\infty.$
On the other hand, by Lemma \ref{measure2}, we deduce that
\begin{align*}
\sum_{1\le m<n\le N}\mathcal{L}(A_m\cap A_n\cap I_t)
&= \sum_{\substack{
            1\le m<n\le N\\
            m=n-1, n-2}}\mathcal{L}(A_m\cap A_n\cap I_t)+\sum_{1\le m<n-2\le N}\mathcal{L}(A_m\cap A_n\cap I_t)\\
&\ll \sum_{n=1}^{N}\mathcal{L}(A_n\cap I_t) +\Big(\sum_{n=1}^{N}\mathcal{L}(A_n\cap I_t)\Big)^{2}|I_t|^{-1}.
 \end{align*}
Thus Chung-Erd\"{o}s inequality yields
$$\mathcal{L}(\mathcal{F}_2(\varphi)\cap I_t)
=\mathcal{L}(\limsup_{n\to\infty} A_n\cap I_t)
\ge \limsup_{N\to \infty}\frac{(\sum_{1\le n\le N}\mathcal{L}(A_n\cap I_t))^2}{\sum_{1\le m\ne n\le N}\mathcal{L}(A_m\cap A_n\cap I_t)}\gg |I_t|.$$
We conclude that $\mathcal{F}_2(\varphi)$ is of full Lebesgue measure in [0,1] by Knopp's lemma.

\section{Proof of Theorem \ref{thm2}: $1<B<\infty$}
We remark that  when $1<B<\infty$,
$$\dim_{H}\mathcal{F}_2(\varphi)=\dim_{H}\mathcal{F}_2(  n\mapsto B^{n}).$$
Therefore, we can simply assume that $\varphi(n)=B^{n}$ and rewrite the set $\mathcal{F}_2(\varphi)$ as
\begin{align*}
\mathcal{F}_2(B)=\Big\{x\in[0,1)\colon  \exists 1\le k< n,  ~&a_{k}(x)a_{k+1}(x)\ge B^{n},
                                                            \\&a_{n}(x)a_{n+1}(x)\ge B^{n} \text{ for infinitely many } n\in \N  \Big\}.
\end{align*}

The aim is to show $\dim_{H}{\mathcal{F}}_2(B)=s_B.$ The proof falls naturally into two parts$\colon$  the upper bound and the lower bound.

\subsection{The upper bound for $\mathcal{F}_2(B)$}
We see that  $\mathcal{F}_2(B)$ is the union of the following two sets$\colon$
\begin{align*}
  \mathfrak{F}_{1}(B)&=\{x\in[0,1)\colon a_{n-1}(x)a_n(x)\ge B^{n},~a_n(x)a_{n+1}(x)\ge B^{n}\text{ for infinitely many } n\in \N\}\\
  \mathfrak{F}_2(B)&=\Big\{x\in[0,1)\colon \begin{array}{c}
                                                  \exists 1\le k< n-1, ~ a_{k}(x)a_{k+1}(x)\ge B^{n}, \\
                                                   a_{n}(x)a_{n+1}(x)\ge B^{n} \text{ for infinitely many } n\in \N \\
                                                 \end{array} \Big\}.
                                                 \end{align*}
 We establish the upper bounds for the dimensions of these two sets.

\begin{lem}\label{up1}
We have that $\dim_{H}\mathfrak{F}_{1}(B)\le s_B.$
\end{lem}

\begin{proof}
Fix $\varepsilon>0$ and let $s=s_B+2\varepsilon$ and $\alpha=B^{1-s}.$
We write $ \mathfrak{F}_{1}(B)=\limsup \mathfrak F_{1,n}$ with
$$\mathfrak F_{1,n}=\{x\in[0,1)\colon a_{n-1}(x)a_n(x)\ge B^{n},~a_n(x)a_{n+1}(x)\ge B^{n}\}. $$
We first cover $\mathfrak F_{1,n}$ by the   following two sets:
\begin{align*}
  \mathfrak{F}_{11,n}& = \{x\in[0,1)\colon a_{n-1}(x)>\alpha^{n}, a_{n-1}(x)a_n(x)\ge B^{n}, a_{n}(x)a_{n+1}(x)\ge B^{n}\},\\
  \mathfrak{F}_{12,n} & =\Big\{x\in[0,1)\colon 1\le a_{n-1}(x)< \alpha^{n},
                                          a_{n-1}(x)a_n(x)\ge B^{n}\Big\}.
\end{align*}
We proceed to cover  $\mathfrak{F}_{11,n}$ by the following sets:
\begin{align*}
  \mathfrak{F}_{111,n}& =\left\{x\in[0,1)\colon a_{n-1}(x)>\alpha^{n}, \frac{B^{n}}{a_{n-1}(x)}\le a_n(x)\le \frac{B^{n}}{\alpha^n}, a_{n+1}(x)\ge\frac{B^n}{a_n(x)}\right\},\\
  \mathfrak{F}_{112,n}& =\left\{x\in[0,1)\colon a_{n-1}(x)>\alpha^{n},  a_n(x)\ge \frac{B^{n}}{\alpha^n}\right\}.
\end{align*}
We continue to cover  $\mathfrak{F}_{111,n}$ and $\mathfrak{F}_{112,n}$ by the collections:
\begin{align*}
 &\Big\{J_{n}(a_1,\ldots,a_{n})=\bigcup_{a_{n+1}\ge \frac{B^n}{a_n}}I_{n+1}(a_1,\ldots,a_{n+1})\colon  a_1,\ldots,a_{n-2}\in\mathbb N, a_{n-1}>\alpha^n, 1\le a_n\le\frac{B^n}{\alpha^n}\Big\},\\
 &\Big\{J'_{n-1}(a_1,\ldots,a_{n-1})=\bigcup_{a_{n}\ge \frac{B^n}{\alpha^n}}I_{n}(a_1,\ldots,a_{n})\colon  a_1,\ldots,a_{n-2}\in\mathbb N, a_{n-1}>\alpha^n\Big\},
\end{align*}
respectively; cover $\mathfrak{F}_{12,n}$ by
\begin{equation*}
\Big\{J''_{n-1}(a_1,\ldots,a_{n-1})=\bigcup_{a_{n}\ge \frac{B^{n}}{a_{n-1}}}I_n(a_1,\ldots, a_{n})\colon
 a_1,\ldots,a_{n-2}\in \mathbb{N}, 1\le a_{n-1}\le \alpha^{n}\Big\}.
\end{equation*}
In  the light of
\begin{equation*}
|J_n|\asymp\frac{1}{B^na_nq_{n-1}^{2}},~
|J'_{n-1}|\asymp\frac{\alpha^n}{B^{n}q_{n-1}^{2}}, ~|J''_{n-1}|\asymp\frac{1}{B^{n}a_{n-1}q_{n-2}^{2}},
\end{equation*}
we estimate the $s$-volume of the above cover of  $\mathfrak{F}_{1,n}\colon$
\begin{align*}
  &\sum_{a_1\ldots,a_{n-2}\in \mathbb{N}}\left[\sum_{a_{n-1}>\alpha^n}\sum_{1\le a_n\le \frac{B^n}{\alpha^n}}\Big(\frac{1}{B^{n}a_{n}q_{n-1}^{2}}\Big)^{s}
  +\sum_{a_{n-1}>\alpha^n}\Big(\frac{\alpha^n}{B^nq_{n-1}^{2}}\Big)^{s}
  +\sum_{1\le a_{n-1}\le \alpha^n}\Big(\frac{1}{B^{n}a_{n-1}q_{n-2}^{2}}\Big)^{s}\right]\\
  &\asymp\sum_{a_1\ldots,a_{n-2}\in \mathbb{N}}\left[B^{n(1-2s)}\alpha^{-ns}\cdot\frac{1}{q_{n-2}^{2s}}
                                                     +B^{-ns}\alpha^{n(1-s)}\cdot\frac{1}{q_{n-2}^{2s}}\right]
  \ \  (\text{integrating on } a_{n-1}, a_n)\\
  &\asymp\sum_{a_1\ldots,a_{n-2}\in \mathbb{N}}\frac{1}{B^{n(3s-1-s^2)}q_{n-2}^{2s}}.
\end{align*}
Hence, we obtain that
\begin{align*}
  \mathcal{H}^{s}( \mathfrak{F}_{1}(B))
  & \le\liminf_{N\to\infty}\sum_{n=N}^{\infty}\sum_{a_1,\ldots,a_{n-2}\in \mathbb{N}}\frac{1}{B^{n(3s-1-s^2)}q_{n-2}^{2s}}\\
  &\le \liminf_{N\to\infty}\sum_{n=N}^{\infty}\sum_{a_1,\ldots,a_{n-2}\in \mathbb{N}}
 \frac{1}{B^{[3(s_B+\varepsilon)-1-(s_B+\varepsilon)^2]}q_{n-2}^{2(s_B+\varepsilon)}}\cdot B^{-3n\varepsilon}\\
 &\ll\liminf_{N\to\infty}\sum_{n=N}^{\infty}B^{-3n\varepsilon}<\infty,
\end{align*}
which follows that $\dim_{H} \mathfrak{F}_{1}(B)\le s_B$.
\end{proof}

\begin{lem}\label{up2} We have that
$\dim_{H} \mathfrak{F}_{2}(B)\le s_B.$
\end{lem}
\begin{proof}We follow the notation of the preceding proof. Write
\begin{align*}
 \mathfrak{F}_{2,n,k} & =\{x\in[0,1)\colon
                                                   a_{k}(x)a_{k+1}(x)\ge B^{n},
                                                   a_{n}(x)a_{n+1}(x)\ge B^{n}\},\\
 \mathfrak{F}_{21,n,k} & = \{x\in[0,1)\colon a_{k}(x)a_{k+1}(x)\ge B^{n}, a_{n}(x)\ge \alpha^{n}\},\\
  \mathfrak{F}_{22,n,k} & =\Big\{x\in[0,1)\colon a_{k}(x)a_{k+1}(x)\ge B^{n}, 1\le a_{n}(x)<\alpha^{n},  a_{n+1}(x)\ge\frac{B^{n}}{a_n(x)}\Big\}.
\end{align*}
Thus, $\mathfrak{F}_2(B) =\bigcap_{N=1}^{\infty}\bigcup_{n=N}^{\infty}\bigcup_{k=1}^{n-2}\mathfrak{F}_{2,n,k} $ with $\mathfrak{F}_{2,n,k} \subset\mathfrak{F}_{21,n,k}\cup  \mathfrak{F}_{22,n,k}$.
We proceed to cover $\mathfrak{F}_{21,n,k}$ by the collection:
\begin{equation*}
 \Big\{J_{n-1}(a_1,\ldots,a_{n-1})=\bigcup_{a_{n}\ge \alpha^n}I_{n}(a_1,\ldots,a_{n})\colon  a_1,\ldots,a_{n-1}\in\mathbb N, a_{k}a_{k+1}\ge B^n\Big\};
\end{equation*}
cover $\mathfrak{F}_{22,n,k}$ by
\begin{equation*}
  \Big\{J'_{n}(a_1,\ldots,a_{n})=\bigcup_{a_{n+1}\ge \frac{B^n}{a_n}}I_{n+1}(a_1,\ldots,a_{n+1})\colon  a_1,\ldots,a_{n}\in\mathbb N, a_{k}a_{k+1}\ge B^n, 1\le a_n\le \alpha^n\Big\}.
\end{equation*}
In the light of
\begin{equation*}
|J_{n-1}|\asymp\frac{1}{\alpha^nq_{n-1}^{2}},~|J'_{n}|\asymp\frac{1}{B^{n}a_nq_{n-1}^{2}}.
\end{equation*}

Therefore, the $s$-volume of the cover of $ \mathfrak{F}_{21,n,k}$ can be
estimate as
\begin{equation*}
\sum_{\substack{
            a_1,\ldots,a_{n-1}\in \N\\
            a_ka_{k+1}\ge B^{n} }} \Big(\frac{1}{\alpha^{n}q_{n-1}^{2}}\Big)^{s}
                                  \asymp n\sum_{a_1,\ldots,a_{n-3}\in \N}B^{n(1-2s)}\alpha^{-ns}\cdot\frac{1}{q_{n-3}^{2s}},
\end{equation*}
and the $s$-volume of the cover of $\mathfrak{F}_{22,n,k}$ is
\begin{equation*}
\sum_{\substack{
                                    a_1,\ldots,a_{n}\in \mathbb{N} \\
                                    a_ka_{k+1}\ge B^{n}, ~1\le a_n< \alpha^{n} }}\Big(\frac{1}{B^{n}a_{n}q_{n-1}^{2}}\Big)^{s}\ll n\sum_{a_1,\ldots,a_{n-3}\in \N}B^{-ns}\alpha^{n(1-s)}\cdot\frac{1}{q_{n-3}^{2s}}.
\end{equation*}
Hence
$$\mathcal{H}^{s}(\mathfrak{F}_{2}(B))\ll\liminf_{N\to\infty}\sum_{n=N}^{\infty}\sum_{k=1}^{n-2}n\sum_{a_1\ldots,a_{n-3}\in \mathbb{N}}\frac{1}{B^{(3s-1-s^2)n}q_{n-3}^{2s}}
\le\liminf_{N\to\infty}\sum_{n=N}^{\infty}B^{-3n\varepsilon}<\infty.$$
\end{proof}

Combing Lemmas \ref{up1} and \ref{up2}, we obtain that
$\dim_{H}\mathcal{F}_2(B)\le s_B.$

\subsection{The lower bound for $\mathcal{F}_2(B)$}
For any integers $M, L,$ we define the pre-dimensional number $S=S_{L,B}(M)$ to be the solution of
$$\sum_{1\le a_1,\ldots,a_L\le M}\frac{1}{B^{(3s-1-s^2)L}q_{L}^{2s}}=1.$$
By Corollary \ref{corollary}, we have $S\to s_B$ as $L, M\to\infty.$

To obtain the lower bound of $\dim_{H}{\mathcal{F}}_2(B)$, we will construct an appropriate Cantor subset of $\mathcal{F}_2(B)$ and then apply
the following mass distribution principle.
\begin{lem} [Falconer, \cite{F}]\label{mass}
Let $E\subseteq [0,1]$ be a Borel set and $\mu$
be a measure with $\mu(E)> 0.$  Suppose there are positive constants $c$ and $r_0$ such that
$$ \mu\big(B(x,r)\big)\le cr^{s}$$
for any ball $B(x,r)$ with radius $r\le r_0$.  Then $ \dim_{H}E\geq s.$
\end{lem}

\subsubsection{Cantor subset construction}
Put $\alpha=B^{1-S}.$ Choose a rapidly increasing sequence  of integers $\{m_k\}_{k\ge1}$ such that $m_0=0$ and $$m_k\ge\frac{64BL^2k^{4}(m_1+\cdots+m_{k-1})}{\log 2}+8kLB^2.$$
   Define another integer sequence $\{n_k\}_{k\ge 0}$ as follows
$$n_0=-1,\ \ n_k=n_{k-1}+m_{k}L+3.$$
We define a desired subset of $\mathcal{F}_2(B)$ as
\begin{align*}
  E=\Big\{x\in[0,1)\colon \alpha^{n_{k}}\le a_{n_k-1}(x)\le&2\alpha^{n_{k}},  
                       \frac{B^{n_k}}{\alpha^{n_k}}\le a_{n_{k}}(x)\le \frac{2B^{n_k}}{\alpha^{n_k}},
                  \alpha^{n_{k}}\le a_{n_{k}+1}(x)\le2\alpha^{n_{k}}
                   \\ &\text{ for } k\ge1
  \text{ and } a_n(x)\in\{1,\ldots,M\} \text{ for other }n\in \N\Big\}.
\end{align*}

We introduce a symbolic space to study the set $E$. Define $D_0=\{\emptyset\}$ and
\begin{align*}
 D_n=\left\{(a_1,\ldots,a_n)\in \mathbb{N}^{n}\colon
                                                        \begin{array}{c}
\alpha^{n_{k}} \le a_{n_k-1}\le2\alpha^{n_{k}},
\frac{B^{n_k}}{\alpha^{n_k}}\le a_{n_k}\le \frac{2B^{n_k}}{\alpha^{n_k}},
\alpha^{n_{k}}\le a_{n_{k}+1}\le2\alpha^{n_{k}}\\  \text{ for } k\ge1 \text{ with } n_k+1\le n,
 \text{ and } a_j\in\{1,\ldots,M\} \text{ for other }j\le n\end{array}\right\}.
\end{align*}
This set is just the collection of the prefixes of the continued fraction expansions of the points in $E.$ For   $n\ge 1$ and $(a_1,\ldots,a_n)\in D_n,$
we call
$$J_n:=J_n(a_1,\ldots,a_n):=\bigcup_{a_{n+1}}I_{n+1}(a_1,\ldots,a_n,a_{n+1})$$
a fundamental interval of order $n,$ where the union is taken over all integers $a_{n+1}$ such that $(a_1,\ldots, a_n, a_{n+1})\in D_{n+1}.$ Then
$$E=\bigcap_{n=1}^\infty\bigcup_{(a_1,\ldots, a_n)\in D_n}J_n(a_1,\ldots, a_n).$$

\subsubsection{Lengths of fundamental intervals}
We calculate the lengths of fundamental intervals by considering several cases.

\noindent\textbf{Case I.} When $n_{k-1}+1\le n\le n_k-3$ for  $k\ge1,$ since
$$J_n(a_1,\ldots,a_n)=\bigcup_{1\le a_{n+1}\le M}I_{n+1}(a_1,\ldots,a_n,a_{n+1}),$$
we have
$$\frac{1}{8q_n^{2}}\le\sum_{a_{n+1}=1}^{M}\frac{1}{2q_{n+1}^{2}} \le |J_n(a_1,\ldots,a_n)|\le\sum_{a_{n+1}=1}^{M}\frac{1}{q_{n+1}^{2}}\le\frac{4}{q_n^{2}},$$
In particular for $n=n_{k-1}+1,$ since
$$\alpha^{n-1}\le a_{n-2}\le2\alpha^{n-1},
\frac{B^{n-1}}{\alpha^{n-1}}\le a_{n-1}\le \frac{2B^{n-1}}{\alpha^{n-1}},
\alpha^{n-1}\le a_{n}\le2\alpha^{n-1},$$
we deduce that
$$\frac{1}{2^{15}(B\alpha)^{2(n-1)}q_{n-3}^{2}}\le |J_n(a_1,\ldots,a_n)|\le\frac{4}{(B\alpha)^{2(n-1)}q_{n-3}^{2}}.$$

\noindent\textbf{Case II.} When $n=n_{k}-2,$ since
$$J_{n}(a_1,\ldots,a_n)=\bigcup_{\alpha^{n+2}\le a_{n+1}\le2\alpha^{n+2}}I_{n+1}(a_1,\ldots,a_n,a_{n+1}),$$
we have
$$|J_{n}(a_1,\ldots,a_n)|\le\sum_{\alpha^{n+2}\le a_{n+1}\le2\alpha^{n+2}}\le\frac{4}{\alpha^{n+2}q_n^{2}},$$
$$|J_{n}(a_1,\ldots,a_n)|\ge\sum_{\alpha^{n+2}\le a_{n+1}\le2\alpha^{n+2}}\frac{1}{4(a_{n+1}q_n)^{2}}\ge\frac{1}{16\alpha^{n+2}q_n^{2}}.$$

\noindent\textbf{Case III.} When $n=n_{k}-1,$ since
$$J_{n}(a_1,\ldots,a_n)=\bigcup_{\frac{B^{n+1}}{\alpha^{n+1}}\le a_{n+1}\le \frac{2B^{n+1}}{\alpha^{n+1}}}I_{n+1}(a_1,\ldots,a_n,a_{n+1}),$$
we have
$$|J_{n}(a_1,\ldots,a_n)|\le \sum_{\frac{B^{n+1}}{\alpha^{n+1}}\le a_{n+1}\le \frac{2B^{n+1}}{\alpha^{n+1}}}\frac{1}{(a_{n+1}q_n)^{2}}\le \frac{4\alpha^{n+1}}{B^{n+1}q_n^{2}},$$
$$|J_{n}(a_1,\ldots,a_n)|\ge \sum_{\frac{B^{n+1}}{\alpha^{n+1}}\le a_{n+1}\le \frac{2B^{n+1}}{\alpha^{n+1}}}\frac{1}{4(a_{n+1}q_n)^{2}}\ge \frac{\alpha^{n+1}}{16B^{n+1}q_n^{2}}.$$

\noindent\textbf{Case IV.} When $n=n_{k},$ since
$$J_{n}(a_1,\ldots,a_n)=\bigcup_{\alpha^{n}\le a_{n+1}\le2\alpha^{n}}I_{n+1}(a_1,\ldots,a_n,a_{n+1}),$$
similar with the \textbf{Case II}, we obtain
$$\frac{1}{16\alpha^{n}q_n^{2}}\le|J_{n}(a_1,\ldots,a_n)|\le\frac{4}{\alpha^{n}q_n^{2}}.$$
Furthermore, noting that
$$\alpha^{n}\le a_{n-1}\le2\alpha^{n},~
\frac{B^{n}}{\alpha^{n}}\le a_{n}\le \frac{2B^{n}}{\alpha^{n}},$$
we have
$$\frac{1}{2^{12}\alpha^{n}B^{2n}q_{n-2}^{2}}\le|J_{n}(a_1,\ldots,a_n)|\le\frac{4}{\alpha^{n}B^{2n}q_{n-2}^{2}}.$$

\subsubsection{Gap estimation}
We estimate the gap between $J_{n}(a_1,\ldots,a_n)$ and its
adjacent fundamental interval of the same order $n$. These   estimations play an important role in  estimating the
measure of general balls.

For $(a_1,\ldots,a_n)\in D_n,$ let $G^l_n(a_1,\ldots, a_n)$ (respectively, $G^r_n(a_1,\ldots, a_n)$) denote  the distance between $J_n(a_1,\ldots, a_n)$ and its left (respectively, right) adjacent fundamental  interval
$$J_n'=J_n'(a_1,\ldots, a_{n-1}, a_n-1) ~\text{(respectively, $J_n''=J_n''(a_1,\ldots, a_{n-1}, a_n+1)$)} \text{ if exists.}$$
And
%
$G_n(a_1,\ldots, a_n)=\min\{G^l_n(a_1,\ldots, a_n), G^r_n(a_1,\ldots, a_n)\}.$
\begin{lem}\label{gap}
For  $n\ge 1$, we have
$$G_n(a_1,\ldots, a_n)\gg |J_n(a_1,\ldots, a_n)|,$$
where the implied constant  depends only upon $M$.
\end{lem}
\begin{proof}
Without loss of generality, we assume that $n$ is even (the estimation for odd $n$  can be carried in almost the
same way). There exists $k\ge 1$ such that $n_{k-1}+1\le n<n_k+1$ for some $k\ge 1$. We consider   four cases.
%

\smallskip

\noindent\textbf{Gap I.} For the case $n_{k-1}+1\le n\le n_k-3,$ we have
$$J_n(a_1,\ldots,a_n)=\bigcup_{1\le a_{n+1}\le M}I_{n+1}(a_1,\ldots,a_n,a_{n+1}).$$
Then by Proposition \ref{cylinder}, the right gap interval between $J_n$ and $J_n''$ is just
$$\bigcup_{a_{n+1}>M}I_{n+1}(a_1,\ldots,a_n+1,a_{n+1})$$
and the left gap interval between $J_n$ and $J_n'$ is
$$\bigcup_{a_{n+1}>M}I_{n+1}(a_1,\ldots,a_n,a_{n+1}).$$
It follows that
\begin{align*}
G_n^{r}(a_1,\ldots,a_n)&=\frac{(M+1)(p_n+p_{n-1})+p_{n-1}}{(M+1)(q_n+q_{n-1})+q_{n-1}}-\frac{p_n+p_{n-1}}{q_n+q_{n-1}}
\\&=\frac{1}{((M+1)(q_n+q_{n-1})+q_{n-1})(q_{n}+q_{n-1})}
\end{align*}
and
$$G_n^{l}(a_1,\ldots,a_n)=\frac{(M+1)p_n+p_{n-1}}{(M+1)q_n+q_{n-1}}-\frac{p_n}{q_n}=\frac{1}{((M+1)q_n+q_{n-1})q_n}.$$
So
$$G_n(a_1,\ldots,a_n)=\frac{1}{((M+1)(q_n+q_{n-1})+q_{n-1})(q_{n}+q_{n-1})}.$$
Comparing $G_n(a_1,\ldots,a_n)$ with $J_n(a_1,\ldots,a_n),$ we obtain
$$G_n(a_1,\ldots,a_n)\ge\frac{1}{40M}|J_n(a_1,\ldots,a_n)|.$$
\smallskip

\noindent\textbf{Gap II.} For the case $ n=n_k-2,$ we have
$$J_n(a_1,\ldots,a_n)=\bigcup_{\alpha^{n+2}\le a_{n+1}\le2\alpha^{n+2}}I_{n+1}(a_1,\ldots,a_n,a_{n+1}).$$
Since  $J_n$ lies in the middle part of $I_n(a_1,\ldots,a_n)$, the gap $G_n^{r}(a_1,\ldots,a_n)$ is at least
the distance between the right endpoint of  $I_n(a_1,\ldots,a_n)$ and that of $J_n(a_1,\ldots,a_n),$ i.e.,
$$G_n^{r}(a_1,\ldots,a_n)\ge \Big|\frac{p_n+p_{n-1}}{q_n+q_{n-1}}
-\frac{(\alpha^{n+2}+1)p_n+p_{n-1}}{(\alpha^{n+2}+1)q_n+q_{n-1}}\Big|
=\frac{\alpha^{n+2}}{((\alpha^{n+2}+1)q_n+q_{n-1})(q_n+q_{n-1})}.$$
In a similar way,
$$G_n^{l}(a_1,\ldots,a_n)
\ge\Big|\frac{(2\alpha^{n+2}+1)p_n+p_{n-1}}{(2\alpha^{n+2}+1)q_n+q_{n-1}}-
\frac{p_n}{q_n}\Big|=\frac{1}{((2\alpha^{n+2}+1)q_n+q_{n-1})q_n}.$$
So
$$G_n(a_1,\ldots,a_n)\ge\frac{1}{((2\alpha^{n+2}+1)q_n+q_{n-1})q_n}
\ge\frac{1}{5\alpha^{n+2}q_n^{2}}\ge\frac{1}{20}|J_n(a_1,\ldots,a_n)|.$$%
\smallskip

\noindent\textbf{Gap III.} For the case $ n=n_k-1,$ we have
$$J_n(a_1,\ldots,a_n)=\bigcup_{\frac{B^{n+1}}{\alpha^{n+1}}\le a_{n+1}\le \frac{2B^{n+1}}{\alpha^{n+1}}}I_{n+1}(a_1,\ldots,a_n,a_{n+1}).$$
As in \textbf{Gap II} we deduce that
$$G_n^{r}(a_1,\ldots,a_n)\ge \frac{\frac{B^{n+1}}{\alpha^{n+1}}}{((\frac{B^{n+1}}{\alpha^{n+1}}+1)q_n+q_{n-1})(q_n+q_{n-1})}$$
and
$$G_n^{l}(a_1,\ldots,a_n)\ge \frac{1}{((\frac{2B^{n+1}}{\alpha^{n+1}}+1)q_n+q_{n-1})q_n}.$$
It follows that
$$G_n(a_1,\ldots,a_n)\ge\frac{1}{((\frac{2B^{n+1}}{\alpha^{n+1}}+1)q_n+q_{n-1})q_n}\ge \frac{1}{20}|J_n(a_1,\ldots,a_n)|.$$

\noindent\textbf{Gap IV.} For the case $ n=n_k,$ we have
$$J_n(a_1,\ldots,a_n)=\bigcup_{\alpha^{n}\le a_{n+1}\le2\alpha^{n}}I_{n+1}(a_1,\ldots,a_n,a_{n+1}).$$
We argue in a similar way to obtain that
$$G_n(a_1,\ldots,a_n)\ge\frac{1}{20}|J_n(a_1,\ldots,a_n)|.$$
\end{proof}

\subsubsection{Mass distribution on $E$}
Put $\beta=B^{(3S-1-S^2)}.$ Recall that
\begin{equation}\label{pre}
\sum_{1\le a_1,\ldots,a_L\le M}\frac{1}{\beta^Lq_{L}^{2S}}=1.\end{equation}
 We distribute a mass on $E$  via assigning the measure for each fundamental intervals as follows.
 For $1\le m\le m_1,$ we define
$$\mu(J_{mL}(a_1,\ldots,a_{mL}))
=\prod_{t=0}^{m-1}\frac{1}{\beta^{L}q_{L}^{2S}(a_{tL+1},\ldots,a_{(t+1)L})}$$
and when $(m-1)L<n<mL,$
$$\mu(J_{n}(a_1,\ldots,a_n))=\sum_{a_{n+1},\ldots,a_{mL}}\mu(J_{mL}(a_1,\ldots,a_n,a_{n+1},\ldots,a_{mL})),$$
where the summation is taken over all $(a_{n+1},\ldots,a_{mL})$ with $(a_{1},\ldots,a_{mL})\in D_{mL}.$

For $n=m_1L+1,$ we define
 $\mu(J_{n_1-1}(a_1,\ldots,a_{n_1-1}))
=\frac{1}{\alpha^{n_1}}\mu(J_{m_1L}(a_1,\ldots,a_{m_1L})).$

For $n=n_1,$ we define
 $\mu(J_{n_1}(a_1,\ldots,a_{n_1}))
=\frac{\alpha^{n_1}}{B^{n_1}}\mu(J_{n_1-1}(a_1,\ldots,a_{n_1-1})).$

For $n=n_1+1,$ we define
 $\mu(J_{n_1+1}(a_1,\ldots,a_{n_1+1}))
=\frac{1}{\alpha^{n_1}}\mu(J_{n_1}(a_1,\ldots,a_{n_1})).$

We then assign inductively the mass on fundamental intervals of other levels:
For $k\ge 2$ and $1\le m\le m_k,$
\begin{align*}
 \mu(J_{n_{k-1}+1+mL}(a_1,\ldots,a_{n_{k-1}+1+mL}))
=&\mu(J_{n_{k-1}+1}(a_1,\ldots,a_{n_{k-1}+1}))\\&
\cdot\prod_{t=0}^{m-1}\frac{1}{\beta^{L}q_{L}^{2S}(a_{n_{k-1}+tL+2},\ldots,a_{n_{k-1}+(t+1)L+1})},
\end{align*}
and when $(m-1)L<n<mL$,
$\mu(J_{n}(a_1,\ldots,a_n))=\sum
\mu(J_{n_{k-1}+1+mL}(a_1,\ldots,a_{n_{k-1}+1+mL})).$

For $n=n_{k-1}+m_{k}L+2,$
 $\mu(J_{n_k-1}(a_1,\ldots,a_{n_k-1}))
=\frac{1}{\alpha^{n_k}}\mu(J_{n_{k-1}+m_{k}L+1}(a_1,\ldots,a_{n_{k-1}+m_{k}L+2})).$

For $n=n_k,$  $\mu(J_{n_k}(a_1,\ldots,a_{n_k}))
=\frac{\alpha^{n_k}}{B^{n_k}}\mu(J_{n_k-1}(a_1,\ldots,a_{n_k-1})).$

For $n=n_k+1,$
 $\mu(J_{n_k+1}(a_1,\ldots,a_{n_k+1}))
=\frac{1}{\alpha^{n_k}}\mu(J_{n_k}(a_1,\ldots,a_{n_k})).$

Have distributed the mass amongst fundamental intervals, for which  the consistency property is ensured by   (\ref{pre}), we  extend it to a measure, denoted still by $\mu$, on all Borel sets.

\subsubsection{H\"{o}lder exponent of $\mu$ for fundamental intervals}
\mbox{}

\underline{\textsc{Case 1$\colon$} $n=mL$ for $1\le m<m_1.$}
\begin{align*}
  \mu(J_{mL}(a_1,\ldots,a_{mL}))
  = & \prod_{t=0}^{m-1}\frac{1}{\beta^{L}q_{L}^{2S}(a_{tL+1},\ldots,a_{(t+1)L})} \\
  \le&4^{m}\frac{1}{q_{mL}^{2S}(a_1,\ldots,a_{mL})}
  \le 32|J_{mL}(a_1,\ldots,a_{mL})|^{S-\frac{2}{L}}.
\end{align*}

\underline{\textsc{Case 2$\colon$} $n=m_1L=n_1-2.$}
\begin{align*}
  \mu(J_{m_1L}(a_1,\ldots,a_{m_1L}))= & \prod_{t=0}^{m_1-1}\frac{1}{\beta^{L}q_{L}^{2S}(a_{tL+1},\ldots,a_{(t+1)L})}
  \le \frac{4^{m_1}}{\beta^{m_1L}}\cdot\frac{1}{q_{m_1L}^{2S}(a_1,\ldots,a_{mL})}\\
  \le&\frac{1}{\alpha^{m_1LS}}\cdot\Big(\frac{1}{q_{m_1L}^{2}(a_1,\ldots,a_{mL})}\Big)^{S-\frac{2}{L}} \ \ (\text{by } \alpha=B^{1-S} \text{ and } S>\frac{1}{2})\\
  \le&\Big(\frac{1}{\alpha^{n+2}}\cdot\frac{1}{q_{n}^{2}(a_1,\ldots,a_n)}\Big)^{S-\frac{3}{L}}
  \le  16|J_{n}(a_1,\ldots,a_{n})|^{S-\frac{3}{L}}.
\end{align*}

\underline{\textsc{Case 3$\colon$} $n=m_1L+1=n_1-1.$}
\begin{align*}
  \mu(J_{n_1-1}(a_1,\ldots,a_{n_1-1}))= & \frac{1}{\alpha^{n_1}}\cdot\mu(J_{n_1-2}(a_1,\ldots,a_{n_1-2}))
  \le\frac{16\alpha^{2n_1S}}{(\alpha\beta)^{n_1}}\Big(\frac{1}{q_{n_1-1}^{2}(a_1,\ldots,a_{n_1-1})}\Big)^{S-\frac{3}{L}}\\
  =&16\Big(\frac{\alpha^{n_1}}{B^{n_1}}\Big)^{S}
  \Big(\frac{1}{q_{n_1-1}^{2}(a_1,\ldots,a_{n_1-1})}\Big)^{S-\frac{3}{L}} 
  \le256|J_{n_1-1}(a_1,\ldots,a_{n_1-1})|^{S-\frac{3}{L}}.
  \end{align*}

\underline{\textsc{Case 4$\colon$} $n=m_1L+2=n_1.$}
\begin{align*}
  \mu(J_{n_1}(a_1,\ldots,a_{n_1}))= &\frac{\alpha^{n_1}}{B^{n_1}}\mu(J_{n_1-1}(a_1,\ldots,a_{n_1-1}))
  =\frac{1}{B^{n_1}}\mu(J_{n_1-2}(a_1,\ldots,a_{n_1-2}))\\
  \le&\frac{1}{B^{n_1}}\cdot\frac{1}{\beta^{n_1}}\cdot\Big(\frac{1}{q_{n_1-2}^{2}(a_1,\ldots,a_{n_1-2})}\Big)^{S-\frac{3}{L}}\\
  =&\Big(\frac{1}{\alpha^{n_1}B^{2n_1}}\Big)^{S}
  \cdot\Big(\frac{1}{q_{n_1-2}^{2}(a_1,\ldots,a_{n_1-2})}\Big)^{S-\frac{3}{L}}~~(\text{by }\alpha=B^{1-S})\\
  \le&2^{12}|J_{n_1}(a_1,\ldots,a_{n_1})|^{S-\frac{3}{L}}.
\end{align*}

\underline{\textsc{Case 5$\colon$} $n=n_1+1.$}
\begin{align*}
 \mu(J_{n_1+1}(a_1,\ldots,a_{n_1+1}))=&\frac{1}{\alpha^{n_1}}\mu(J_{n_1}(a_1,\ldots,a_{n_1}))
                                      =\frac{1}{(B\alpha)^{n_1}}\mu(J_{n_1-2}(a_1,\ldots,a_{n_1-2}))\\
 \le&\frac{1}{(B\alpha\beta)^{n_1}}
  \cdot\Big(\frac{1}{q_{n_1-2}^{2}(a_1,\ldots,a_{n_1-2})}\Big)^{S-\frac{3}{L}}\\
  \le&\Big(\frac{1}{(B\alpha)^{2n_1}}\Big)^{S}
  \cdot\Big(\frac{1}{q_{n_1-2}^{2}(a_1,\ldots,a_{n_1-2})}\Big)^{S-\frac{3}{L}}\\
  \le&2^{15}|J_{n_1+1}(a_1,\ldots,a_{n_1+1})|^{S-\frac{3}{L}}.
\end{align*}

\underline{\textsc{Case 6$\colon$} general case.}

We  give the estimation only for $J_{n_k-2}(a_1,\ldots,a_{n_k-2}).$
The other cases can be carried out in the same way. %
\begin{align*}
&\mu(J_{n_{k}-2})=\mu(J_{n_{k-1}+1})
 \cdot\prod_{t=0}^{m_k-1}\frac{1}{\beta^{L}q_{L}^{2S}(a_{n_{k-1}+tL+2},\ldots,a_{n_{k-1}+(t+1)L+1})}\\
 =&\left[\prod_{j=1}^{k-1}
  \left(\frac{1}{(B\alpha)^{n_j}}
  \prod_{t=0}^{m_j-1}\Big(\frac{1}{\beta^{L}q_{L}^{2S}(a_{n_{j-1}+tL+2},\ldots,a_{n_{j-1}+(t+1)L+1})}\Big)\right)\right]\\
 &\cdot\prod_{t=0}^{m_k-1}\frac{1}{\beta^{L}q_{L}^{2S}(a_{n_{k-1}+tL+2},\ldots,a_{n_{k-1}+(t+1)L+1})}\\
  \le& \prod_{j=1}^{k-1}
 \Big(\frac{1}{(B\alpha)^{n_j}} \cdot\frac{2^{m_j}}{\beta^{m_jL}q_{m_jL}^{2S}(a_{n_{j-1}+2},\ldots,a_{n_{j}-2})}\Big)
  \cdot\frac{2^{m_k}}{\beta^{m_kL}q_{m_kL}^{2S}(a_{n_{k-1}+2},\ldots,a_{n_{k}-2})}\\
  \le &\prod_{j=1}^{k-1}\Big(\frac{2^{m_j}\beta^{n_{j-1}+3}}{(B\alpha\beta)^{n_j}q_{m_jL}^{2S}(a_{n_{j-1}+2},\ldots,a_{n_{j}-2})}\Big)
  \cdot\frac{\beta^{2^{m_k}n_{k-1}+3}}{\beta^{n_k}q_{m_kL}^{2S}(a_{n_{k-1}+2},\ldots,a_{n_{k}-2})}\\
  \le& (2B)^{n_1+\cdots+n_{k-1}+8k}\Big(\frac{1}{q^{2}_{n_{k-1}+1}}\Big)^{S-\frac{3}{L}}
  \cdot\Big(\frac{1}{\alpha^{n_k}q_{m_kL}^{2}(a_{n_{k-1}+2},\ldots,a_{n_{k}-2})}\Big)^{S-\frac{3}{L}}
  \le\Big(\frac{1}{\alpha^{n_k}q_{n_k-2}^{2}}\Big)^{S-\frac{4}{L}}.
\end{align*}

In summary,  we have shown that for any $n\ge 1$ and $(a_1,\ldots,a_n)\in D_n,$
$$\mu(J_n(a_1,\ldots,a_n))\ll |J_n(a_1,\ldots,a_n)|^{S-\frac{4}{L}}.$$

\subsubsection{The H\"{o}lder exponent for a general ball}
We are in a position to estimate the measure $\mu$ of a general ball $B(x,r).$

Let $B(x,r)$ be a ball with $x\in E$ and with radius $r$ small enough. There exists
$n\in\mathbb N$ and    $(a_1, a_2, \ldots, a_n)\in D_n$ such that $x\in J_n(a_1,\ldots,a_n)$   and
 $G_{n+1}(a_1,\dots,a_{n+1})\le r<G_n(a_1,\dots,a_n).$ Whence $B(x,r)$ can intersect only one fundamental interval of order $n,$
namely $J_{n}(a_1,\ldots,a_n).$
\smallskip

\noindent\textbf{Case I.} $n_{k-1}+1\le n<n_k-2$ for some $k\ge1.$
Then $1\le a_{n+1}\le M$, $|J_n|\asymp\frac{1}{q_n^{2}},$ and
\begin{align*}
\mu(B(x,r))\le \mu(J_n)\ll|J_{n}|^{S-\frac{4}{L}}\ll (\frac{1}{q_{n+1}^{2}})^{S-\frac{4}{L}}\ll|J_{n+1}|^{S-\frac{4}{L}}\ll G_{n+1}^{S-\frac{4}{L}}\ll r^{S-\frac{4}{L}},
\end{align*}
where the penultimate inequality follows from Lemma \ref{gap}.

\smallskip

\noindent\textbf{Case II.} $n=n_k-2$ for some $k\ge1.$

(1) $r\le |I_{n_{k}-1}(a_1,\ldots,a_{n_{k}-1})|.$

In this case, the ball $B(x,r)$ can intersect at most four cylinders
of order $n_{k}-1,$ which are  $I_{n_{k}-1}(a_1,\ldots,a_{n_{k}-1}+i)$ with $i=-1,0,1,2$.
Thus we have
\begin{align*}
  \mu(B(x,r)) &\le 4\mu(J_{n_{k}-1})
  \ll |J_{n_{k}-1}|^{S-\frac{4}{L}}\ll G_{n+1}^{S-\frac{4}{L}} \ll r^{S-\frac{4}{L}}.
\end{align*}

(2) $r>|I_{n_{k}-1}(a_1,\ldots,a_{n_{k}-1})|.$

 In this case, since
 $|I_{n_{k}-1}(a_1,\ldots,a_{n_{k}-1})|\ge\frac{1}{2q_{n_{k}-1}^{2}}\ge \frac{1}{8\alpha^{2n_k}q_{n_{k}-2}^{2}},$
 in $J_{n_{k}-2}(a_1,\ldots,a_{n_{k}-2})$ there are at most
$$16r\alpha^{2n_k}q_{n_{k}-2}^{2}+2 \le 32r\alpha^{2n_k}q_{n_{k}-2}^{2}$$
number of fundamental intervals of order $n_{k}-1$ intersecting $B(x,r)$, and thus
\begin{align*}
  \mu(B(x,r)) &\le \min\Big\{\mu(J_{n_{k}-2}), 32r\alpha^{2n_k}q_{n_{k}-2}^{2} \mu(J_{n_{k}-1})\Big\}   \le \mu(J_{n_{k}-2})\min\Big\{1, 32r\alpha^{2n_k}q_{n_{k}-2}^{2}\cdot\frac{1}{\alpha^{n_k}}\Big\}\\
  &\ll |J_{n_{k}-2}|^{S-\frac{2}{L}} \min\Big\{1, 32r\alpha^{n_k}q_{n_{k}-2}^{2}\Big\}  \le \Big(\frac{1}{\alpha^{n_k}q_{n_k-2}^{2}}\Big)^{S-\frac{4}{L}}
  \Big(32r\alpha^{n_k}q_{n_{k}-2}^{2}\Big)^{S-\frac{4}{L}} \ll r^{S-\frac{4}{L}}.
\end{align*}
Here we use the fact that if $a,b>0$ and  $0<s<1$, then $\min\{a,b\}\le a^{1-s}b^s$.

\noindent\textbf{Case III.} $n=n_k-1$ or $n=n_{k}$ for some $k\ge1.$
\smallskip

Similar argument to \textbf{Case II} shows that $\mu(B(x,r))\ll r^{S-\frac{4}{L}}.$
\medskip

Combining all cases we assert that $\mu(B(x,r))\ll r^{S-\frac{4}{L}}$ for any $x\in E$ and $r$ small. The mass distribution principle (Lemma \ref{mass}) yields that
$$\dim_{H}E\ge S-\frac{4}{L}=S_{L,B}(M)-\frac{4}{L}.$$
Letting $L, M\to\infty,$ we conclude that $\dim_{H}{\mathcal{F}}_2(B)\ge s_B.$

\medskip

{\noindent \bf  Acknowledgements}. This work was supported by NSFC No. 12171172, 12201476. 
The authors thank Hussain Mumtaz and Feng Jing for helpful discussions.

\end{document}